\documentclass[preprint,12pt]{elsarticle}

\usepackage{amssymb}
\usepackage{amsthm}

\newcommand{\sysn}{\left\{\begin{array}{rcl}}
\newcommand{\sysk}{\end{array}\right.}

\newtheorem{theorem}{Theorem}[section]

\theoremstyle{example}
\newtheorem{example}[theorem]{Example}

\newtheorem{proposition}[theorem]{Proposition}
\theoremstyle{definition}
\newtheorem{definition}[theorem]{Definition}
\newtheorem{corollary}[theorem]{Corollary}

\journal{Topology and its Applications}

\begin{document}

\begin{frontmatter}


%%\title{Topological-Algebraic Properties of Function Space with Set-Open Topology\tnoteref{label1}}
%%\tnotetext[label1]{}
%%\author{Alexander V. Osipov\corref{cor1}\fnref{label2}}
%%\ead{OAB@list.ru}

\title{On extension functions for image space with different separation axioms  \tnoteref{label1}}


\author{Alexander V. Osipov}

\ead{OAB@list.ru}

\tnotetext[label1]{The research has been supported by Act 211
Government of the Russian Federation, contract ¹ 02.A03.21.0006.}

\address{Institute of Mathematics and Mechanics, Ural Branch of the Russian Academy of Sciences, 16, S.Kovalevskaja street,
620219, Ekaterinburg, Russia     \\  Ural Federal University}

\begin{abstract}
In this paper we study a sufficient conditions for continuous and
$\theta_{\alpha}$-continuous extensions of $f$ to space $X$ for an
image space $Y$ with different separation axioms.
\end{abstract}


\end{frontmatter}

% \linenumbers


\section{Introduction}
\label{}

It is work was performed as part of the general problem, which is
as follows. Let $f$ be a continuous mapping of a dense set $S$ of
the topological space $X$ into the topological space $Y$. Required
to find the necessary and sufficient conditions for continuous
extension of $f$ to the space $X$. (i.e. existence a continuous
mapping $F:X \mapsto Y$ such that $F\upharpoonright S=f$). This
problem can be considered more widely, if the continuous mapping
is replaced be "almost" continuous. For example, we will consider
the $\theta_{\alpha}$-continuous mapping.

First sufficient condition for continuous extension of $f$ to the
space $X$ into the regular space $Y$ was obtained by N. Bourbaki.
In {\cite{bur}} was proved that this condition is not sufficient
condition for no regular space $Y$.

The necessary and sufficient conditions for continuous extension
of $f$ on the space $X$ was obtained:
\medskip

in ~{\cite{vul}}  for metrizable compact spaces $Y$;
\medskip

in ~{\cite{ta}} for compact spaces $Y$;
\medskip

in ~{\cite{vel}} for Lindel$\ddot{e}$of spaces $Y$;
\medskip

in ~{\cite{eng}} for realcompact spaces $Y$;
\medskip

in ~{\cite{vel}} for regular spaces $Y$.

\medskip

So for a compact spaces $Y$ we have the next result (see
~{\cite{ta}}).

\begin{theorem}(Taimanov)\label{th0} Let $f$ be a continuous mapping of a dense set $S$ of a
topological space $X$ into a compact space $Y$, then  the
following are equivalent:

\begin{enumerate}

\item $f$ to have a continuous extension to $X$;

\item  if $A$ and $B$ are disjoint closed subsets of $Y$ then
$\overline{f^{-1}(A)}\bigcap \overline{f^{-1}(B)}=\emptyset$.

\end{enumerate}
\end{theorem}

Consider the following

$\bullet $ Condition $(*)$: if a family $\{A_{\beta}\}$ of closed
subsets of $Y$ such that $\bigcap_{\beta} A_{\beta}=\emptyset$
implies $\bigcap_{\beta} \overline{f^{-1}(A_{\beta})}=\emptyset$.

\medskip

So general result for a regular space $Y$ is the following theorem
(\cite{vel}).

\begin{theorem}(Velichko)\label{th1} Let $f$ be a continuous mapping of a dense set $S$ of a
topological space $X$ into a regular space $Y$, then  the
following are equivalent:

\begin{enumerate}

\item $f$ to have a continuous extension to $X$;

\item  condition $(*)$ holds.

\end{enumerate}
\end{theorem}

Note that if $Y$ is a Tychonoff space, then we can in condition
$(*)$ a closed subsets be replaced by zero-sets of $Y$.

\begin{theorem}(Velichko)\label{th1} Let $f$ be a continuous mapping of a dense set $S$ of the
topological space $X$ into a Lindel$\ddot{e}$of space $Y$, then
the following are equivalent:

\begin{enumerate}

\item $f$ to have a continuous extension to $X$;

\item  for any sequence $\{A_{i}\}$ of zero-sets of $Y$ such that
$\bigcap_{i} A_{i}=\emptyset$ implies $\bigcap_{i}
\overline{f^{-1}(A_{i})}=\emptyset$.

\end{enumerate}
\end{theorem}

Note that the condition $(*)$ is a necessary condition for
continuous extension of $f$ to $X$ for any space $Y$.

\begin{proposition}\label{pr11} Let $f$ be have a continuous extension to $X$ for
a space $Y$. Then condition $(*)$ holds.
\end{proposition}

\begin{proof} Let $F$ be a continuous extension to $X$ for a
a space $Y$ and  $\{A_{\beta}\}$ be a family of closed subsets of
$Y$ such that $\bigcap_{\beta} A_{\beta}=\emptyset$. Fix $x\in X$.
There is $\beta$ such that $F(x)\notin A_{\beta}$, hence there
exist a neighborhood $V$ of $F(x)$ such that $V\bigcap
A_{\beta}=\emptyset$. Since $F$ is continuous map, $F^{-1}(V)$ is
a neighborhood of $x$. It follows that $F^{-1}(V)\bigcap
F^{-1}(A_{\beta})=\emptyset$ and $x\notin
\overline{f^{-1}(A_{\beta})}$.
\end{proof}

In this paper we study a sufficient conditions for continuous and
$\theta_{\alpha}$-continuous extensions of $f$ to $X$ for an image
space $Y$ with different separation axioms.

 P.S. Alexandroff and P.S. Urysohn ~{\cite{au}} first defined $R$-closed spaces in
 1924.
\medskip
$U$-closed ($R$-closed) spaces are Urysohn (regular) spaces,
closed in any Urysohn (regular) space containing them. Recall that
the solutions to the standard extension of continuous functions
problem in the setting of $R$-closed or $U$-closed spaces are
unknown.

\medskip

${\bf Question}$ 1 (Q.22 in ~{\cite{fgpp}}). Let $X$ be an
$R$-closed extension of a space $S$ and $f:S\mapsto Y$ be a
continuous function where $Y$ is $R$-closed. Find a necessary and
sufficient condition for $f$ to have a continuous extension to
$X$.

\medskip

${\bf Question}$ 2 (Q.23 in ~{\cite{fgpp}}). Let $X$ be an
$U$-closed extension of a space $S$ and $f:S\mapsto Y$ be a
continuous function where $Y$ is $U$-closed. Find a necessary and
sufficient condition for $f$ to have a continuous extension to
$X$.

\medskip

In this paper we find a necessary and sufficient condition for $f$
to have a continuous extension to $X$ where $Y$ is $U(\alpha)$-
space (regular $U(\alpha)$-space).

\section{Main definitions and notation}

We say that $U$ is a neighborhood of a set $A$ if $U$ is  an open
set in $X$ such that $A\subseteq U$.

 The closure of a set $A$ will be denoted by $\overline{A}$ , $[A]$ or $cl(A)$; the symbol
 $\varnothing$ stands for the empty set. As usual, $f(A)$ and $f^{-1}(A)$ are the image and
 the complete preimage of the set $A$ under the mapping~$f$,
 respectively.

  Let $\alpha>0$ be an ordinal.

\begin{definition} A neighborhood $U$ of a set $A$ is called an $\alpha$-hull of
the set $A$ if there exists a set of neighborhoods
$\{U_{\beta}\}_{\beta\leq\alpha}$ of the set $A$ such that
$clU_{\beta}\subseteq U_{\beta+1}$ for any  $\beta+1\leq\alpha$
and $U=U_{\alpha}=\bigcup_{\beta\leq\alpha} U_{\beta}$.
\end{definition}

For $\alpha=1$,  a $1$-hull of the set $A$ is an open set
containing the set $A$.

Let $X$ be a topological space, $M\subseteq X$, $x\in X$ and
$\alpha>0$ be an ordinal. We consider the
$\theta^{\alpha}$-closure operator: $x\notin cl_{\theta^{\alpha}}
M$ if there is an $\alpha$-hull $U$ of the point $x$ such that
$clU\bigcap M=\emptyset$ if $\alpha>1$; $cl_{\theta^0}M=clM$ if
$\alpha=0$; and, for $\alpha=1$, we get the $\theta$-closure
operator, i.e., $cl_{\theta^1}M=cl_{\theta}M$.

 A set $M$ is
$\theta^{\alpha}$-closed if $M=cl_{\theta^{\alpha}}M$. Denote by
$Int_{\theta^{\alpha}}M=X\setminus cl_{\theta^{\alpha}}(X\setminus
M)$ the $\theta^{\alpha}$-interior of the set $M$. Evidently,
$cl_{\theta^{\alpha_1}}(cl_{\theta^{\alpha_2}}M)=cl_{\theta^{\alpha_1+\alpha_2}}M$
for $M\subseteq X$ and any $\alpha_1$, $\alpha_2$ ordinal numbers.

For $\alpha>0$ and a filter $\mathcal F$ on $X$, denote by
$ad_{\theta^{\alpha}}\mathcal F$ the set of
$\theta^{\alpha}$-adherent points, i.e.,
$ad_{\theta^{\alpha}}\mathcal F=\bigcap \{
cl_{\theta^{\alpha}}\mathcal F_{\beta} : F_{\beta}\in \mathcal
F\}$. In particular, $ad_{\theta^0}\mathcal F=ad\mathcal F$ is the
set of adherent points of the filter $\mathcal F$. For any
$\alpha$, a point $x\in X$ is $S(\alpha)$-separated from a subset
$M$ if $x\notin cl_{\theta^{\alpha}}M$. For example, $x$ is
$S(0)$-separated from $M$ if $x\notin clM$. For $\alpha>0$, the
relation of $S(\alpha)$-separability of points is symmetric. On
the other hand, $S(0)$-separability may be not symmetric in some
not $T_1$-spaces. Therefore, we say that points $x$ and $y$ are
$S(0)$-separated if $x\notin cl_{X}\{y\}$ and $y\notin
cl_{X}\{x\}$.

Let $n\in \mathbb{N}$ and $X$ be a topological space.

1. $X$ is called an $S(n)$-space if any two distinct points of $X$
are $S(n)$-separated.

Next, we define a series of separation axioms for $\alpha>0$.

2. $X$ will be called an $U(\alpha)$-space if any two distinct
points $x$ and $y$ of $X$ there are $U_{x}$, $U_{y}$ an
$\alpha$-hull of $x$ and $y$ such that $\overline{U_{x}}\bigcap
\overline{U_{y}}=\emptyset$.

\medskip
Note that a regular space is a $S(n)$-($U(n)$-)space for any $n\in
\mathbb{N}$, and a functionally Hausdorff space is a
$U(\omega)$-space.

3. $X$ will be called a regular-$U(\alpha)$-space if  for a point
$x$ and a closed  set $F$ such that $x\notin F$ there are $U_{x}$,
$U_{F}$ an $\alpha$-hull of $x$ and $F$ such that
$\overline{U_{x}}\bigcap \overline{U_{F}}=\emptyset$.

Note that a Tychonoff space is a regular-$U(\omega)$-space.

It is obvious that $S(0)$-spaces are $T_1$-spaces, $S(1)$-spaces
are Hausdorff spaces, and $S(2)$-spaces are Urysohn spaces.

\medskip
A set of all of neighborhoods of $x$ will be denoted by
$\mathcal{N}(x)$.

\medskip

A set of all $\alpha$-hull of the set $A$ will be denoted by
$\mathcal{N}_{\theta^{\alpha}}(A)$.

\medskip

\section{Continuous extension}

\begin{definition} Let $X, Y$ be a topological spaces, $S$ be a dense subset of $X$, $f$ be a  continuous map from $S$ into
$Y$ and $V$ be a subset of $Y$. A point $x\in X$ will be called
{\it $X_{\theta^{\alpha}}$-interior point} of $f^{-1}(V)$, if

$\bigcap \{[f(P\bigcap S)]_{\theta^{\alpha}} : P\in
\mathcal{N}(x)\}\subseteq V$ holds.

\end{definition}

Set of all $X_{\theta^{\alpha}}$-interior points of $f^{-1}(V)$
will be denoted by $X_{\theta^{\alpha}}(f^{-1}(V))$.

\bigskip

\begin{proposition}\label{pr1}
 Let $X$ be a topological space, $Y$ be a $U(\alpha)$-space, $S$ be a dense subset of $X$, $f$ be a  continuous map from $S$ into
$Y$, $V$ be a open subset of $Y$, $x\in X$ and the condition $(*)$
holds. Then the set $\bigcap \{[f(P\bigcap S)]_{\theta^{\alpha}} :
P\in \mathcal{N}(x)\}=\{p\}$ for some $p\in Y$.
\end{proposition}

\begin{proof} Note that $x\in\bigcap \{\overline{P\bigcap S} :
P\in \mathcal{N}(x)\}$. By condition $(*)$,  $\bigcap
\{\overline{f(P\bigcap S)} : P\in \mathcal{N}(x)\}\neq \emptyset$
and, hence, $T=\bigcap \{[f(P\bigcap S)]_{\theta^{\alpha}} : P\in
\mathcal{N}(x)\}\neq \emptyset$. Let $y\in T$ and $z\neq y$. By
$U(\alpha)$-separateness of $Y$, there are $\alpha$-hulls $O(y)$
and $O(z)$ such that $\overline{O(y)}\bigcap
\overline{O(z)}=\emptyset$. Then
$\overline{f^{-1}(\overline{O(y)})}\bigcap
\overline{f^{-1}(\overline{O(z)})}=\emptyset$. Note, that $x\in
\overline{f^{-1}(\overline{O(y)})}$. Then there is $W\in
\mathcal{N}(x)$ such that $W\bigcap
\overline{f^{-1}(\overline{O(z)})}=\emptyset$. It follows that
$f(W\bigcap S)\bigcap \overline{O(z)}=\emptyset$ and $z\notin
[f(W\bigcap S)]_{\theta^{\alpha}}$.

\end{proof}

$\bullet $ Condition $(*_{\alpha})$: for each an open set $V$ of
$Y$ the set $X_{\theta^{\alpha}}(f^{-1}(V))$ is open set of $X$.

\bigskip

\begin{theorem}
 Let $X$ be a topological space, $Y$ be a $U(\alpha)$-space, $S$ be a dense subset of $X$, $f$ be a
continuous map from $S$ into $Y$, then  the following are
equivalent:

\begin{enumerate}

\item $f$ to have a continuous extension to $X$;

\item  conditions $(*)$ and $(*_{\alpha})$ holds.

\end{enumerate}

\end{theorem}

\begin{proof}
$(1)\Rightarrow(2)$. From Proposition \ref{pr11}, we obtain
condition $(*)$. Let $f$ to have a continuous extension $F$ to $X$
and let $V$ be an open set of $Y$. We prove that
$X_{\theta^{\alpha}}(f^{-1}(V))=F^{-1}(V)$.

Let $x\in X_{\theta^{\alpha}}(f^{-1}(V))$. Since $F$ is a
continuous map and the filter base $\{P\bigcap S: P\in
\mathcal{N}(x)\}$ converges to a point $x$, it follows that the
filter base $\mathcal F=\{f(P\bigcap S): P\in \mathcal{N}(x)\}$
converges to a point $F(x)$. By $U(\alpha)$-separateness of $Y$,
$ad_{\theta^{\alpha}}\mathcal F=F(x)$, hence,  $F(x)\in V$ and
$x\in F^{-1}(V)$.

Let $x\in F^{-1}(V)$. Then $F(x)\in V$. By unique of adherent
point of the filter base $\mathcal F=\{f(P\bigcap S): P\in
\mathcal{N}(x)\}$, we have $ad_{\theta^{\alpha}}\mathcal F=F(x)\in
V$ and $x\in X_{\theta^{\alpha}}(f^{-1}(V))$.

$(2)\Rightarrow(1)$. For each point $x\in X$ consider
$F(x)=\bigcap \{[f(P\bigcap S)]_{\theta^{\alpha}} : P\in
\mathcal{N}(x)\}$. By  proposition  ~\ref{pr1}, $F(x)$ is an
unique point of $Y$. We have the map $F: X \mapsto Y$. Note that
if $x\in S$ then $F(x)=f(x)$. Clearly that $x\in P\bigcap S$ for
any $P\in \mathcal{N}(x)$, so $f(x)\in \bigcap \{f(P\bigcap S) :
P\in \mathcal{N}(x)\}\subseteq \bigcap \{[f(P\bigcap
S)]_{\theta^{\alpha}} : P\in \mathcal{N}(x)\}=F(x)$. We have that
$F$ is an extension $f$ on $X$. We claim that $F$ is a continuous
extension on $X$. Let $V$ be an open set of $Y$. By condition
$(*_{\alpha})$, $X_{\theta^{\alpha}}(f^{-1}(V))$ is open set of
$X$. It remains to prove that
$X_{\theta^{\alpha}}(f^{-1}(V))=F^{-1}(V)$.

Let $x\in X_{\theta^{\alpha}}(f^{-1}(V))$. Then, by condition
$(*_{\alpha})$, $F(x)=\bigcap \{[f(P\bigcap S)]_{\theta^{\alpha}}
: P\in \mathcal{N}(x)\}\subseteq V$, and $x\in F^{-1}(V)$.

Let $x\in F^{-1}(V)$. Then $F(x)\in V$ and $\bigcap \{[f(P\bigcap
S)]_{\theta^{\alpha}} : P\in \mathcal{N}(x)\}=F(x)\in V$, thus the
point $x$ is $X_{\theta^{\alpha}}$-interior point of $f^{-1}(V)$,
hence, $x\in X_{\theta^{\alpha}}(f^{-1}(V))$. It follow that
$X_{\theta^{\alpha}}(f^{-1}(V))=F^{-1}(V)$.

\end{proof}

\begin{corollary}

 Let $X$ be a topological space, $Y$ be a  Urysohn space, $S$ be a dense subset of $X$, $f$ be a
continuous map from $S$ into $Y$, then  the following are
equivalent:

\begin{enumerate}

\item $f$ to have a continuous extension to $X$;

\item for each an open set $V$ of $Y$ the set $X(f^{-1}(V))$ is
open set of $X$ and condition $(*)$ holds.

\end{enumerate}

\end{corollary}

\begin{corollary}

 Let $X$ be a topological space, $Y$ be a  functionally Hausdorff space, $S$ be a dense subset of $X$, $f$ be a
continuous map from $S$ into $Y$, then  the following are
equivalent:

\begin{enumerate}

\item $f$ to have a continuous extension to $X$;

\item conditions $(*)$ and $(*_{\omega})$ holds.

\end{enumerate}

\end{corollary}

\begin{proposition}\label{pr2}
Let $Y$ be a regular-$U(\alpha)$-space. Then $(*)$ implies
$(*_{\alpha})$.

\end{proposition}

\begin{proof} Let $S$ be a dense subset of $X$, $f: S\mapsto Y$ be a continuous function and condition $(*)$ holds.
We prove that a set $X_{\theta^{\alpha}}(f^{-1}(V))$ is open set
of $X$ for an open set $V$ of $Y$.  Let $x\in
X_{\theta^{\alpha}}(f^{-1}(V))$. By proposition ~\ref{pr1},
$F(x)=\bigcap \{[f(P\bigcap S)]_{\theta^{\alpha}} : P\in
\mathcal{N}(x)\}$ is an unique point of $Y$. As $Y$ is a
regular-$U(\alpha)$-space there is $\alpha$-hull $W$ of point
$F(x)$ and $\alpha$-hull $H$ of set $X\setminus V$
 such that $\overline{W}\bigcap \overline{H}=\emptyset$. Let
$\gamma=\{TP=[f(P\bigcap S)\bigcap (X\setminus
\overline{W})]_{\theta^{\alpha}} : P\in \mathcal{N}(x)\}$. Then
$\bigcap \gamma=\emptyset$ and $\bigcap \{f^{-1}(TP): P\in
\mathcal{N}(x)\}=\emptyset$. There is a neighborhood $U$ of $x$
such that $U\bigcap f^{-1}(TP)=\emptyset$ for some $P\in
\mathcal{N}(x)$. Let $Q=U\bigcap P$. Then $[f(P\bigcap
S)]_{\theta^{\alpha}}\subseteq \overline{W}$. It follows that
$Q\subseteq X_{\theta^{\alpha}}(f^{-1}(V))$ and, hence, the set
$X_{\theta^{\alpha}}(f^{-1}(V))$ is an open set of $X$.

\end{proof}

\begin{theorem}\label{th15}
 Let $X$ be a topological space, $Y$ be a regular-$U(\alpha)$-space, $S$ be a dense subset of $X$, $f$ be a
continuous map from $S$ into $Y$, then  the following are
equivalent:

\begin{enumerate}

\item $f$ to have a continuous extension to $X$;

\item  conditions $(*)$ holds.

\end{enumerate}

\end{theorem}

\section{$\theta_{\alpha}$-continuous extension}

Recall that a function $f: X \mapsto Y $  be called
$\theta$-continuous if for a point $x\in X$ and a neighborhood $U$
of $f(x)$ there is a neighborhood $W$ of $x$ such that
$f(W)\subseteq \overline{U}$.

\begin{definition} A function $f: X \mapsto Y $ will be called
$\theta_{\alpha}$-continuous if for a point $x\in X$ and a
$\alpha$-hull $U$ of $f(x)$ there is a neighborhood $W$ of $x$
such that $f(W)\subseteq \overline{U}$.
\end{definition}

\bigskip
For $\alpha=1$, we have that $\theta_{1}$-continuous function is
$\theta$-continuous function.

 Clearly, that a continuous function is a
$\theta_{\alpha}$-continuous function for any $\alpha>0$.
Moreover, it is easy to check that $\theta_{\beta}$-continuous
function is $\theta_{\alpha}$-continuous function for
$\beta<\alpha$.
\bigskip

$\bullet $ Condition $(+)_{\alpha}$: a family $\{A_{\beta}\}$ of
subsets of $Y$ such that $\bigcap_{\beta}
[A_{\beta}]_{\theta^{\alpha}}=\emptyset$ implies $\bigcap_{\beta}
\overline{f^{-1}(A_{\beta})}=\emptyset$.

\bigskip

$\bullet$ Condition $(++)_{\alpha}$: for each $\alpha$-hull
$W=\bigcup_{\beta\leq\alpha} U_{\beta}$ of a point $y\in Y$ there
is an open set $V$ of $X$ such that
$X_{\theta^{\alpha}}(f^{-1}(U_1))\subseteq V \subseteq
X_{\theta^{\alpha}}(f^{-1}(\overline{W}))$.

\bigskip

\begin{theorem}
 Let $X$ be a topological space, $Y$ be a $U(\alpha)$-space, $S$ be a dense subset of $X$, $f$ be a
$\theta_{\alpha}$-continuous map from $S$ into $Y$, then  the
following are equivalent:

\begin{enumerate}

\item $f$ to have a $\theta_{\alpha}$-continuous extension to $X$;

\item  conditions $(+)_{\alpha}$ and $(++)_{\alpha}$ holds.

\end{enumerate}

\end{theorem}

\begin{proof}
$(1)\Rightarrow(2)$. Let $F$ be a $\theta_{\alpha}$-continuous
extension of $f$, a family $\sigma=\{A_{\beta}\}$ such that
$\bigcap_{\beta} [A_{\beta}]_{\theta^{\alpha}}=\emptyset$, $x\in
X$ and $y=F(x)$. There is $B\in \sigma$ such that $y\notin
[B]_{\theta^{\alpha}}$ and, hence, there is a $\alpha$-hull $W$ of
$y$ such that $\overline{W}\bigcap B=\emptyset$. There exists a
neighborhood $V$ of $x$ such that $F(V)\subseteq \overline{W}$.
Since $\overline{W}\bigcap B=\emptyset$, we get $V\bigcap
f^{-1}(B)=\emptyset$, and, hence, $x\notin \overline{f^{-1}(B)}$.
It follows that $\bigcap \{\overline{f^{-1}(B)}: B\in \sigma
\}=\emptyset$. So we have condition $(+)_{\alpha}$ holds.

 Let $W=\bigcup_{\beta\leq\alpha} U_{\beta}$ (where $clU_{\beta}\subseteq U_{\beta+1}$ for any  $\beta+1\leq\alpha$ and
$W=U_{\alpha}=\bigcup_{\beta\leq\alpha} U_{\beta}$) be an
$\alpha$-hull of some point of $Y$ and $x\in X$ such that

$x\in X_{\theta^{\alpha}}(f^{-1}(U_1))$. Since $x\in \bigcap
\{\overline{P\bigcap S} : P\in \mathcal{N}(x)\}$ we get $F(x)\in
\bigcap \{[f(P\bigcap S)]_{\theta^{\alpha}}: P\in
\mathcal{N}(x)\}$. If $a\neq F(x)$ then there are a $\alpha$-hull
$O(a)$ and $O(F(x))$ such that $\overline{O(a)}\bigcap
\overline{O(F(x))}=\emptyset$. Then there is a neighborhood $P$ of
$x$ such that $F(P)\subseteq \overline{O(F(x))}$. It follow that
$a\notin [f(P\bigcap S)]_{\theta^{\alpha}}$ and $\bigcap
\{[f(P\bigcap S)]_{\theta^{\alpha}}: P\in \mathcal{N}(x)\}=F(x)$.
Since $F(x)\in U_1$ (and $W$ is $\alpha$-hull of $F(x)$) there is
a neighborhood $V_x$ of $x$ such that $F(V_x)\subseteq
\overline{W}$. So if $z\in V_x$ then $F(z)\in \overline{W}$ and,
hence, $V_x \subseteq X_{\theta^{\alpha}}(f^{-1}(\overline{W}))$.
Let $V=\bigcup \{V_{x} : x\in X_{\theta^{\alpha}}(f^{-1}(W))\}$.

So we have condition $(++)_{\alpha}$ holds.

$(2)\Rightarrow(1)$. Let $F(x):=\bigcap \{[f(P\bigcap
S)]_{\theta^{\alpha}}: P\in \mathcal{N}(x)\}$. By condition
$(+)_{\alpha}$, $F(x)\neq \emptyset$. We claim that $F(x)$ is an
unique point. Let $y\in F(x)$ and $z\neq y$. Then there are
$\alpha$-hull $O(y)$ and $O(z)$ of points $y$ and $z$ such that
$\overline{O(y)}\bigcap \overline{O(z)}=\emptyset$. Let $P$ be a
neighborhood of $x$. Then $z\notin [f(P\bigcap S)\bigcap
\overline{O(y)}]_{\theta^{\alpha}}$. We claim that $z\notin F(x)$.
On the contrary, let $z\in F(x)$. Then a family
$\gamma=\{f(P\bigcap S)\bigcap \overline{P(z)} : P\in
\mathcal{N}(x), P(z)\in \mathcal{N}_{\theta^{\alpha}}(z)\}$
consists of a non-empty sets and $y\notin [f(P\bigcap S)\bigcap
\overline{O(z)}]_{\theta^{\alpha}}$. Consider a family
$\sigma=\gamma\bigcup \{f(P\bigcap S)\bigcap \overline{P(y)} :
P\in \mathcal{N}(x), P(y)\in \mathcal{N}_{\theta^{\alpha}}(y)\}$.

 We claim that $D=\bigcap \{[B]_{\theta^{\alpha}} : B\in
\sigma\}=\emptyset$. So $y\notin D$ and $z\notin D$. Let $q\in
Y\setminus \{y,z\}$. There are $\alpha$-hull $P(y)$ and $P(q)$ of
points $y$ and $q$ such that $\overline{P(y)}\bigcap
\overline{P(q)}=\emptyset$. Then $q\notin [f(P\bigcap S)\bigcap
\overline{P(y)}]_{\theta^{\alpha}}$ and $q\notin D$. So
$D=\emptyset$.

By condition $(+)_{\alpha}$, $\bigcap \{\overline{f^{-1}(B)} :
B\in \sigma \}=\emptyset$. Hence, there is $C\in \sigma$ such that
$x\notin \overline{f^{-1}(C)}$. Note that $x\in
\overline{f^{-1}(B)}$ for any $B\in (\sigma\setminus\gamma)$ (by
definition of $F(x)$). It follows that $C\in \gamma$, i.e.
$C=f(P\bigcap S)\bigcap \overline{P(z)}$ for some $P\in
\mathcal{N}(x)$ and $P(z)\in \mathcal{N}_{\theta^{\alpha}}(z)$.
There is a neighborhood $Q$ of $x$ such that $Q\bigcap
f^{-1}(C)=\emptyset$ and $Q\subseteq P$. So $f(Q\bigcap S)\bigcap
\overline{P(z)}=\emptyset$ and $z\notin F(x)$.

So we get extension $F$ of the map $f$. We claim that $F$ is a
$\theta_{\alpha}$-continuous extension to $X$.

Let $x\in X$ and $W=\bigcup_{\beta\leq\alpha} U_{\beta}$ be a
$\alpha$-hull of $F(x)$. Then $F(x)=\bigcap \{[f(P\bigcap
S)]_{\theta^{\alpha}}: P\in \mathcal{N}(x)\}$ and $F(x)\in U_1$.
Hence $x\in X_{\theta^{\alpha}}(f^{-1}(U_1))$. By condition
$(++)_{\alpha}$, there is an open set $V$ of $X$ such that
$X_{\theta^{\alpha}}(f^{-1}(U_1))\subseteq V \subseteq
X_{\theta^{\alpha}}(f^{-1}(\overline{W}))$. It follows that
$F(V)\subseteq \overline{W}$.

\end{proof}

Note that for $\alpha=1$ we get

$\bullet $ condition $(+)$: a family $\{A_{\beta}\}$ of subsets of
$Y$ such that $\bigcap_{\beta} [A_{\beta}]_{\theta}=\emptyset$
implies $\bigcap_{\beta} \overline{f^{-1}(A_{\beta})}=\emptyset$.

$\bullet$ condition $(++)$: for each open set $W$ of $X$ there is
an open set $V$ of $X$ such that $X_{\theta}(f^{-1}(W))\subseteq V
\subseteq X_{\theta}(f^{-1}(\overline{W}))$.

\medskip

\begin{corollary}

 Let $X$ be a topological space, $Y$ be a Urysohn space, $S$ be a dense subset of $X$, $f$ be a
$\theta$-continuous map from $S$ into $Y$, then  the following are
equivalent:

\begin{enumerate}

\item $f$ to have a $\theta$-continuous extension to $X$;

\item  conditions $(+)$ and $(++)$ holds.

\end{enumerate}

\end{corollary}

\medskip

Note that for a regular-$U(\alpha)$-space $Y$ a
$\theta_{\alpha}$-continuous function is a continuous function and
$cl_{\theta^{\alpha}}M=cl M$. It follows that a condition
$(+)_{\alpha}$ is equivalent to the condition $(*)$ and we get a
Theorem \ref{th15}.

\section{Example}

There is a simple example of a regular space, but it is not
completely regular space (see~{\cite{my}}).

\begin{example}(Mysior) Let $M_0$ be the subset of the plane defined by
the condition $y\geq 0$, i.e., the closed upper half-plane, let
$z_0$ be the point $(0,-1)$ and let $M=M_0\bigcup \{z_0\}$. Denote
by $L$ the line $y=0$ and by $L_i$ where $i=1,2,...,$ the segment
consisting of all points $(x,0)\in L$ with $i-1\leq x \leq i$. For
each point $z=(x,0)\in L$ denote by $A_1(z)$ the set of all points
$(x,y)\in M_0$, where $0\leq y\leq 2$, by $A_2(z)$ the set of all
points $(x+y,y)\in M_0$, where $0\leq y\leq 2$, and let $B(z)$ be
the family of all sets of the form $(A_1(z)\bigcup
A_2(z))\setminus B$, where $B$ is a finite set such that $z\notin
B$. Furthermore, for each point $z\in M_0\setminus L$ let
$B(z)=\{\{z\}\}$ and, finally, let
$B(z_0)=\{U_i(z_0)\}_{i=1}^{\infty}$, where $U_i(z_0)$ consists of
$z_0$ and all points $(x,y)\in M_0$ with $x\geq i$.

\end{example}

It is well-known that the space $M$ is a regular space, but it is
not a Tychonoff space. Moreover, the space $M$ is not regular
$U(\omega)$-space.

Let $T$ be the space $M$, but a base of the point $z_0$ we define
as $B(z_0)=M\setminus D$, where $D$ a clopen subset of $M$ such
that $z_0\notin D$. Note that the identity map $id: M \mapsto T$
is the Tychonoff functor.

Consider a continuous identity map $f: M_0 \mapsto M_0$ as the map
from a dense subset $M_0$ of the space $T$ into the space $M$.

1. $f$ have not a $\theta_{\alpha}$-continuous extension to $X$
for any $\alpha<\omega$.

Really, the set $L_i$ is a $\theta_{\alpha}$-closed subset of $M$
for any $i\in \mathbb{N}$ and $\alpha<\omega$. Clearly that
$\bigcap L_i=\emptyset$, but $\bigcap_i
\overline{f^{-1}(L_i)}=\{z_0\}$. This contradicts the condition
$(+)_{\alpha}$ for $\alpha<\omega$.

2. $f$ have a $\theta_{\omega}$-continuous extension to $X$.

Let $F=id: T \mapsto M$. Note that for a $\omega$-hull $W$ of
$z_0$ of the space $M$, $L\subseteq W$. Hence, $F^{-1}(W)$ is an
open set of $T$.

\medskip

% {\bf Acknowledgment.}

%This work was supported by Act 211 Government of the Russian
%Federation, contract ¹ 02.A03.21.0006.



%% References


\bibliographystyle{model1a-num-names}
\bibliography{<your-bib-database>}


%%\bibliographystyle{plain}




% \end{thebibliography}

\end{document}